\documentclass[12pt,reqno]{amsart}
\usepackage{amsmath, amsthm, amssymb,booktabs}

\advance \topmargin by -\headheight
\advance \topmargin by -\headsep
     
\evensidemargin \oddsidemargin
\newtheorem{theorem}{Theorem}
\newtheorem{lemma}[theorem]{Lemma}

\theoremstyle{definition}

\theoremstyle{remark}

\numberwithin{equation}{section}

\newcommand{\mmod}[1]{\,\,(\text{mod}\,\,#1)}

\def\lam{{\lambda}}

\def\le{\leqslant} \def\ge{\geqslant}

\begin{document}
\title[A superpowered Euclidean prime generator]{A superpowered Euclidean prime generator}
\author[Trevor D. Wooley]{Trevor D. Wooley}
\address{School of Mathematics, University of Bristol, University Walk, Clifton, 
Bristol BS8 1TW, United Kingdom}
\email{matdw@bristol.ac.uk}

\begin{abstract} When $k\ge 1$ and $n$ is the product of the smallest $k$ primes, the 
$(k+1)$-st smallest prime is the least divisor exceeding $1$ of $n^{n^n}-1$. This variant 
of Euclid's prime generator is discussed with some of its brethren.
\end{abstract}
\maketitle

When $\{p_1,\ldots ,p_k\}$ is a finite set of primes, the least divisor exceeding $1$ of 
$p_1\cdots p_k+1$ is a prime distinct from $p_1,\ldots ,p_k$. In this way, as every 
schoolchild knows, one sees that there are infinitely many primes: the assumption that 
there are just finitely many leads to a contradiction. This is essentially the proof attributed 
to Euclid, who observed that all primes dividing $p_1\cdots p_k+1$ are distinct from 
$p_1,\ldots ,p_k$. But is {\it every} prime delivered by iterating this algorithm? To be 
precise, if we put $p_1=2$ and then define
$$p_{k+1}=\min \left\{ d>1:\text{$d$ divides $p_1\cdots p_k+1$}\right\}\quad (k\ge 1),
$$
is it the case that the sequence $(p_k)_{k=1}^{\infty}$ contains all the primes? 
The widely held conjecture that the answer is in the affirmative remains open more than 
half a century after Mullin \cite{Mul1963} posed this question. There are, however, variants 
of Euclid's construction that do yield every prime. Given a set of primes 
$\{p_1,\ldots ,p_k\}$, Pomerance \cite[\S1.1.3]{CP2005} defines $p_{k+1}$ 
to be the least prime distinct from $p_1,\ldots ,p_k$ that divides a number of the form 
$d+1$ for some divisor $d$ of $n=p_1\cdots p_k$. He shows that starting with $p_1=2$, 
every prime is delivered by this iterative process, and moreover (by extensive 
computations) that $p_k$ is the $k$-th smallest prime for $k\ge 5$. Booker 
\cite{Boo2016} instead considers the prime divisors $p$ of the integers $d+n/d$, and 
shows that at each stage in the iteration, choices for $d$ and $p$ may be made so that, 
taking $p_{k+1}=p$, every prime is delivered.\par

The iterative processes of Booker \cite{Boo2016} and Pomerance \cite{CP2005} involve 
some kind of ambiguity, in the latter case involving a choice of the divisor $d$ 
of $n=p_1\cdots p_k$, and in the former case a choice of both $d$ and the prime divisor 
of $d+n/d$. In this note we present a variant of Euclid's prime generator 
in which the sequence of primes is determined in order by a single choice of divisor.

\begin{theorem}\label{theorem1} Let $p_1=2$, and when $k\ge 1$, define $p_{k+1}$ to 
be the least divisor exceeding $1$ of $n^{n^n}-1$, where $n=p_1\cdots p_k$. 
Then for each $k$, the integer $p_k$ is the $k$-th smallest prime.
\end{theorem}

This ``superpowered'' variant of Euclid's prime generator has computational value that can 
only be described as rather less than nanoscopic. However, it has the merit of 
succinctly delivering the $(k+1)$-st smallest prime in terms of the $k$ smallest primes. The 
proof is immediate from the following lemma, the proof of which is reminiscent of the 
argument underlying the Pollard $p-1$ factorization method (see \cite{Pol1974} and 
\cite[\S5.4]{CP2005}).

\begin{lemma}\label{lemma2} When $n$ is an integer exceeding $1$, the least prime 
divisor of $n^{n^n}-1$ is the smallest prime not dividing $n$.
\end{lemma}   

\begin{proof} Let $\pi$ be the smallest prime not dividing $n$, and let the primes 
dividing $n$ be $p_1,\ldots ,p_k$. Then one has $\pi\le p_1\cdots p_k+1\le n+1$, as 
Euclid could have told us. Moreover, all prime divisors of $\pi-1$ lie in 
$\{p_1,\ldots ,p_k\}$, and since $p_i^n\ge 2^n\ge n+1\ge \pi$ for each $i$, we find that 
$\pi-1$ divides $(p_1\cdots p_k)^n$, and hence also $n^n$. But then, defining the integer 
$\lam$ by writing $n^n=\lam (\pi-1)$, and noting that $\pi$ does not divide $n$, we find 
from Fermat's Little Theorem that $n^{n^n}=(n^\lam)^{\pi-1}\equiv 1\mmod{\pi}$, 
which is to say that $\pi$ divides $n^{n^n}-1$.
\end{proof}

The argument just described makes it apparent that less profligate exponents are viable. 
The conclusion of Theorem \ref{theorem1} remains valid, for example, when 
$n^{n^n}-1$ is replaced by $n^{n^m}-1$, in which 
$m=\lceil (\log n)/(\log 2)\rceil$. In this context, we note 
also that if $p_k$ denotes the $k$-th smallest prime for each $k$, and 
$n=p_1\cdots p_k$, then the argument of the proof of Lemma \ref{lemma2} shows that 
{\it all} of the primes $p$ with $p_{k+1}\le p<2p_{k+1}$ divide $n^{n^n}-1$.\par

A Euclidean disciple even more orthodox than enthusiasts of Theorem \ref{theorem1} 
might demand a means of obtaining the next prime without knowing a single one 
of the previous (smaller) primes. Even zero-knowledge demands such as this can be 
met by a direct consequence of Lemma \ref{lemma2}. 

\begin{theorem}\label{theorem3} When $N$ is an integer exceeding $1$, the smallest 
prime exceeding $N$ is the least divisor exceeding $1$ of $N!^{N!^{N!}}-1$.
\end{theorem}

For a proof, simply apply Lemma \ref{lemma2} with $n=N!$. We encourage readers to 
entertain themselves by establishing that for each natural number $N$, the smallest prime 
exceeding $N$ is the least divisor exceeding $1$ of $N!^{N!}-1$ (the author is grateful to 
Andrew Booker and Andrew Granville for pointing out this refinement).

\providecommand{\bysame}{\leavevmode\hbox to3em{\hrulefill}\thinspace}


\begin{thebibliography}{18}

\bibitem{Boo2016}
A. R. Booker, \emph{A variant of the Euclid-Mullin sequence containing every prime}, 
arXiv:1605.08929v1.

\bibitem{CP2005}
R. Crandall and C. Pomerance, \emph{Prime numbers: a computational perspective}, 2nd 
ed., Springer, New York, 2005.

\bibitem{Mul1963}
A. A. Mullin, \emph{Research Problem 8. Recursive function theory}, Bull. Amer. Math. Soc. 
\textbf{69} (1963), no. 6, 737.

\bibitem{Pol1974}
J. M. Pollard, \emph{Theorems on factorization and primality testing}, Math. Proc. 
Cambridge Philos. Soc. \textbf{76} (1974), no. 3, 521--528.

\end{thebibliography}
\end{document}